\documentclass[a4paper,10pt,twoside,reqno]{amsart}

\usepackage{enumerate}
\usepackage{amsmath}
\usepackage{amssymb}
\usepackage{amsfonts}
\usepackage{amsthm}

\newtheorem{thm}{Theorem}[section]
\newtheorem{lem}[thm]{Lemma}
\newtheorem{prop}[thm]{Proposition}
\newtheorem{cor}[thm]{Corollary}

\theoremstyle{definition}
\newtheorem{dfn}[thm]{Definiton}
\newtheorem{example}[thm]{Example}

\theoremstyle{remark}
\newtheorem*{rem}{Remark}

\newcommand{\set}[1]{\{#1\}}
\newcommand{\al}{\alpha}
\newcommand{\be}{\beta}
\newcommand{\ga}{\gamma}
\newcommand{\de}{\delta}
\newcommand{\e}{\epsilon}
\newcommand{\N}{\mathbb{N}}

\newcommand{\diam}{\operatorname{diam}}

\newcommand{\m}{\mathbf{m}}

\newcommand{\cO}{\mathcal{O}}

\numberwithin{equation}{section}

\begin{document}

\title[Fixed point theorems for Meir-Keeler type contractions]%
{Fixed point theorems for Meir-Keeler type contractions in metric spaces}

\author[M. Abtahi]{Mortaza Abtahi}

\address{School of Mathematics and Computer Sciences,
Damghan University, Damghan, P.O.BOX 36715-364, Iran}

\email{abtahi@du.ac.ir; mortaza.abtahi@gmail.com}



\begin{abstract}
  We establish a simple and powerful lemma that provides a criterion for
  sequences in metric spaces to be Cauchy. Using the lemma, it is then easily
  verified that the Picard iterates $\{T^nx\}$, where
  $T$ is a contraction or asymptotic contraction of Meir-Keeler type,
  are Cauchy sequences.
  As an application, new and simple proofs for
  several known results on the existence of a fixed point for continuous
  and asymptotically regular self-maps of complete metric spaces satisfying
  a contractive condition of Meir-Keeler type are derived. These results include
  the remarkable fixed point theorem of Proinov in
  [Petko D. Proinov, Fixed point theorems in metric spaces,
  Nonlinear Anal. \textbf{46} (2006) 546--557],
  the fixed point theorem of Suzuki for asymptotic contractions
  in [Tomonari Suzuki, A definitive result on asymptotic contractions,
  J. Math. Anal. Appl. \textbf{335} (2007) 707--715], and others.
  We also prove some new fixed point theorems.
\end{abstract}

\maketitle

\section{Introduction}
\label{sec:intro}

Metric fixed point theory is a very extensive area of analysis with various applications.
Many of the most important nonlinear problems of applied mathematics reduce to
finding solutions of nonlinear functional equations which can be formulated
in terms of finding the fixed points of a given nonlinear operator of an infinite
dimensional function space $X$ into itself. There is a classical general existence
theory of fixed points for mappings satisfying a variety of contractive conditions.
The first basic result is the Banach contraction principle.

\begin{thm}[Banach \cite{Banach}]
\label{thm:Banach}
  Let $X$ be a complete metric space and $T:X\to X$ be a strict contraction,
  that is, there exists $\al$, $0<\al<1$, such that
  \begin{equation}\label{eqn:Banach contraction}
    \forall x,y\in X,\quad d(Tx,Ty) \leq \al d(x,y).
  \end{equation}
  Then $T$ has a unique fixed point $z$, and $T^nx\to z$, for every $x\in X$.
\end{thm}

The Banach contraction principle is fundamental in fixed point theory.
It has been extended to some larger classes of contractive mappings
by replacing the strict contractive condition \eqref{eqn:Banach contraction}
by weaker conditions of various types;  see, for example,
\cite{Boyd-Wong-1969,
Kannan-1969,
Singh-1969,
Ciric-1971,
Bianchini-1972,
Sehgal-1972,
Chatterjea-1972,
Hardy-Rogers-1973,
Ciric-1974,
Caristi-1976,
Ekeland-1974,
Subrahmanyam-1974,
Ciric-1981}.
A comparative study of some of these results have been made
by Rhoades \cite{Rhoades-1977}.

There are thousands of theorems which assure the existence of a fixed point of
a self-map $T$ of a complete metric space $X$. These theorems can be categorized
into different types, \cite{Suzuki-new-type-of-fp-thm-2009}. One type, and perhaps
the most common one, is called the Leader-type \cite{Leader-1983}: the mapping $T$
has a unique fixed point, and the fixed point can always be found by using
Picard iterates $\set{T^nx}$, beginning with some initial choice $x$ in $X$.
Most of the theorems belong to Leader-type.
For instance, \'Ciri\'c in \cite{Ciric-1974} defined the class of \emph{quasi-contractions}
on a metric space $X$ consisting of all mappings $T$ for which there
exists $\al$, $0<\al<1$, such that $d(Tx, Ty) \leq \al \m(x,y)$, for every $x,y\in X$, where
\begin{equation*}
  \m(x,y)=\max\set{d(x, y), d(x, Tx), d(y, Ty), d(x, Ty), d(y, Tx)}.
\end{equation*}
The results presented in \cite{Ciric-1974} show that the condition of
quasi-contractivity implies all conclusions of Banach contraction principle.
We remark that \'Ciri\'c's quasi-contraction is considered as the most general
among contractions listed in \cite{Rhoades-1977}.

Another interesting generalization of Banach contraction principle was
given, in 1969, by Meir and Keeler \cite{Meir-Keeler-1969}.
They defined \emph{weakly uniformly strict contraction} mappings and proved
a fixed point theorem that generalized the fixed point theorem of Boyd and Wong \cite{Boyd-Wong-1969}
and extended the principle to wider classes of maps than those covered in \cite{Rhoades-1977}.

\begin{dfn}\label{dfn:Meir-Keeler-contraction}
  A mapping $T$ on a metric space $X$ is said to be a \emph{Meir-Keeler contraction}
  (or a \emph{weakly uniformly strict contraction} \cite{Meir-Keeler-1969})
  if, for every $\e>0$, there exists $\de>0$ such that
  \begin{equation*}
    \forall\, x,y\in X,\quad
    \e \leq d(x,y) < \e+\de\ \Longrightarrow\
    d(Tx,Ty)<\e.
  \end{equation*}
\end{dfn}

\begin{thm}[Meir and Keeler \cite{Meir-Keeler-1969}]
\label{thm:Meir-Keeler}
  Let $X$ be a complete metric space and let $T$ be a Meir-Keeler contraction
  on $X$. Then $T$ has a unique fixed point $z$, and $T^nx\to z$, for every $x\in X$.
\end{thm}

The Meir and Keeler's generalized version of Banach contraction principle
initiated a lot of work in this direction and led to some important contribution
in metric fixed point theory; see, for example,
\cite{Maiti-Pal-1978,
Park-Rhoades-1981,
Rao-Rao-1985,
Jachymski-1995,
Cho-Murthy-Jungck-2000}.
The following theorem of \'Ciri\'c \cite{Ciric-1981} and Matkowski \cite[Theorem 1.5.1]{Kuczma}
generalizes the above Meir-Keeler fixed point theorem.

\begin{dfn}
\label{dfn:Ciric-Matkowski contraction}
  A mapping $T$ on a metric space $X$ is said to be
  a \emph{\'Ciri\'c-Matkowski contraction} if $d(Tx,Ty)<d(x,y)$ for
  every $x,y\in X$ with $x\neq y$, and, for every $\e>0$, there exists $\de>0$ such that
  \begin{equation}\label{eqn:Ciric-Matkowski-contraction}
    \forall x,y\in X, \quad
    \e< d(x,y) < \e+\de \Longrightarrow d(Tx,Ty) \leq \e.
  \end{equation}
\end{dfn}

Obviously, the class of \'Ciri\'c-Matkowski contractions contains the class
of Meir-Keeler contractions. As it is mentioned in \cite[Proposition 1]{Jachymski-1995},
it is easy to see that condition \eqref{eqn:Ciric-Matkowski-contraction} in Definition \ref{dfn:Ciric-Matkowski contraction} can be replaced by the following:
\begin{equation*}
  \forall\, x,y\in X,\quad
  d(x,y) < \e+\de\ \Longrightarrow\
  d(Tx,Ty)\leq \e.
\end{equation*}

\begin{thm}[\'Ciri\'c \cite{Ciric-1981}, Matkowski \cite{Kuczma}]
\label{thm:Ciric-Matkowski}
  Let $X$ be a complete metric space and let $T$ be a \'Ciri\'c-Matkowski contraction
  on $X$. Then $T$ has a unique fixed point $z$, and $T^nx\to z$, for every $x\in X$.
\end{thm}

In 1995, Jachymski \cite[Theorem 2]{Jachymski-1995} replaced the distance function $d(x, y)$ in
the \'Ciri\'c-Matkowski theorem by the following:
\begin{equation*}
  \m(x, y) = \max \set{d(x, y), d(x, T x), d(y, Ty), [d(x, Ty) + d(y, T x)]/2}.
\end{equation*}
As in \cite{Proinov-2006}, we refer to this result of Jachymski as
the Jachymski-Matkowski theorem because it is equivalent to a result
of Matkowski \cite[Theorem 1]{Matkowski-1980}.

In 2006, extending \'Ciri\'c's quasi-contraction to a very general setting,
Proinov \cite{Proinov-2006} obtained the following remarkable fixed point theorem
generalizing Jachymski-Matkowski theorem.

\begin{dfn}
\label{dfn:contractive+asymptotically-regular}
  A self-map $T$ of a metric space $X$ is said to be
  \emph{contractive} \cite{Edelstein-62} if $d(Tx, Ty) < d(x,y)$,
  for all $x,y\in X$ with $x \neq y$; it is called
  \emph{asymptotically regular} \cite{Browder-1966} if
  $d(T^nx, T^{n+1}x)\to0$, for each $x\in X$.
\end{dfn}

\begin{thm}[Proinov \cite{Proinov-2006}]
\label{thm:Proinov}
  Let $T$ be a continuous and asymptotically regular self-map
  of a complete metric space $X$. Fix $\ga\geq0$, and define
  \begin{equation}\label{eqn:Proinov-m(x,y)}
    \m(x,y) = d(x, y) + \ga[d(x, T x) + d(y, Ty)].
  \end{equation}
  Suppose $d(Tx, Ty) < \m(x, y)$ for all $x, y\in X$ with $x\neq y$,
  and, for any $\e>0$, there exists $\de>0$ such that $\m(x, y) < \de+\e$
  implies $d(T x, Ty)\leq \e$.
  Then $T$ has a unique fixed point $z$, and the Picard iterates of $T$
  converge to $z$.
\end{thm}

After establishing a technical lemma in section \ref{sec:technical-lemma},
we present, in section \ref{sec:fixed-point-theorems}, a fixed point theorem
that generalizes the Proinov's Theorem \ref{thm:Proinov}.
We shall also discuss asymptotic contractions of Meir-Keeler type
in section \ref{sec:asymptotic-contractions}. The significance
of our results is their simple proofs despite their generality.

\medskip \noindent
\textbf{Convention.}
Since we are mainly concerned with Picard iterates $\set{T^nx}_{n=0}^\infty$ of a given self-map $T$,
it is more convenient to take $\N_0=\N\cup\set{0}$ as the indexing set of all sequences
in this paper.

\bigskip

\section{A Technical Lemma}
\label{sec:technical-lemma}

The lemma we present in this section is fundamental in our discussion.
It provides a criterion for sequences in metric spaces to be Cauchy.
As a result, it can be easily verified that, if $T$ is a contraction
of Meir-Keeler type, then the Picard iterates of $T$
are Cauchy sequences.

\begin{lem}
\label{lem:technical-lemma}
  Let $\set{x_n}$ be a sequence in a metric space. If $d(x_n,x_{n+1})\to0$,
  then the following condition implies that $\{x_n\}$ is a Cauchy sequence.
  \begin{itemize}
    \item for every $\e>0$, there exists a sequence $\set{\nu_n}$
    of nonnegative integers such that, for any two subsequences
    $\{x_{p_n}\}$ and $\{x_{q_n}\}$, if\/ $\limsup d(x_{p_n},x_{q_n})\leq \e$,
    then, for some $N$,
    \begin{equation}\label{eqn:d(xpn+nun,xqn+nun)<=e}
      d(x_{p_n+\nu_n},x_{q_n+\nu_n}) \leq \e, \qquad (n\geq N).
    \end{equation}
  \end{itemize}
\end{lem}

It should be mentioned that the following proof of the lemma actually stems from the work of
Geraghty in \cite{Geraghty-73}.

\begin{proof}
  Assume, towards a contradiction, that $\set{x_n}$ is not a Cauchy sequence.
  Then, there exists $\e>0$ such that
  \begin{equation}\label{eqn:negation-of-Cauchy}
    \forall k\in\N,\ \exists\, p,q\geq k,
    \quad d(x_p,x_q)>\e.
  \end{equation}
  For this $\e$, let $\{\nu_n\}$ be the sequence of
  nonnegative integers given by the assumption.
  Since $d(x_n,x_{n+1})\to0$, there exist positive integers $k_1<k_2<\dotsb$ such that
  \[
    d(x_\ell,x_{\ell+1})<\frac1{n(\nu_n+1)}, \qquad (\ell\geq k_n).
  \]
  For each $k_n+\nu_n$, by \eqref{eqn:negation-of-Cauchy}, there exist
  integers $s_n$ and $t_n$ such that $t_n>s_n\geq k_n+\nu_n$ and
  $d(x_{s_n},x_{t_n})>\e$. We let $t_n$ be the smallest such integer
  so that $d(x_{s_n},x_{t_n-1})\leq\e$.
  Take $p_n=s_n-\nu_n$ and $q_n=t_n-\nu_n$. Then $q_n>p_n \geq k_n$, and
  \[
    d(x_{p_n+\nu_n},x_{q_n+\nu_n})>\e
    \quad \text{and} \quad
    d(x_{p_n+\nu_n},x_{q_n+\nu_n-1}) \leq \e.
  \]
  Using triangle inequality, we have, for every $n$,
  \begin{equation*}
    d(x_{p_n},x_{q_n})
    \leq \frac{\nu_n}{n(\nu_n+1)}+d(x_{p_n+\nu_n},x_{q_n+\nu_n-1}) +\frac{\nu_n+1}{n(\nu_n+1)}.
  \end{equation*}
  This implies that $\limsup d(x_{p_n},x_{q_n}) \leq \e$. Since
  $d(x_{p_n+\nu_n},x_{q_n+\nu_n})>\e$, for every $n$,
  we get a contradiction.
\end{proof}

\begin{rem}
  In the above proof of the lemma, if we
  apply the triangle inequality for the second time, we get
  \[
    \e < d(x_{p_n+\nu_n},x_{q_n+\nu_n})
     \leq \frac{\nu_n}{n(\nu_n+1)}+d(x_{p_n},x_{q_n}) +\frac{\nu_n}{n(\nu_n+1)}.
  \]
  This implies that $\e\leq\liminf d(x_{p_n},x_{q_n})$.
  Hence we have $d(x_{p_n},x_{q_n})\to\e$.
\end{rem}

Using the following theorem, and its successive corollary, we will be able to give
very simple proofs for theorems mentioned in section \ref{sec:intro}.

\begin{thm}
\label{thm:main}
  Let $X$ be a metric space and let $\{x_n\}$ be a sequence in $X$.
  Suppose $\m$ is a nonnegative function on $X\times X$ such that,
  for any two subsequences $\{x_{p_n}\}$ and $\{x_{q_n}\}$,
  \begin{equation}\label{eqn:limsup m <= limsup d}
      \limsup_{n\to\infty} \m(x_{p_n},x_{q_n})
      \leq \limsup_{n\to\infty} d(x_{p_n},x_{q_n}).
  \end{equation}
  If $d(x_n,x_{n+1})\to0$, then the following condition implies that
  $\set{x_n}$ is Cauchy.
  \begin{itemize}
    \item for every $\e>0$, there exists a sequence $\set{\nu_n}$
    of nonnegative integers such that, for any two subsequences
    $\{x_{p_n}\}$ and $\{x_{q_n}\}$, if\/ $\limsup \m(x_{p_n},x_{q_n})\leq \e$,
    then, for some $N\in\N$,
    \begin{equation*}
      d(x_{p_n+\nu_n},x_{q_n+\nu_n}) \leq \e, \qquad (n\geq N).
    \end{equation*}
  \end{itemize}
\end{thm}

\begin{proof}
  Let $\e>0$ and let $\{\nu_n\}$ be the sequence of nonnegative integers given by
  the assumption. Let $\{x_{p_n}\}$ and $\{x_{q_n}\}$ be two subsequences with
  $\limsup d(x_{p_n},x_{q_n})\leq \e$. Then $\limsup \m(x_{p_n},x_{q_n})\leq \e$,
  and thus \eqref{eqn:d(xpn+nun,xqn+nun)<=e} holds. All conditions in
  Lemma \ref{lem:technical-lemma} are fulfilled and so the sequence
  is Cauchy.
\end{proof}

If $\{\nu_n\}$ is a constant sequence, e.g.\ $\nu_n=\nu$ for all $n$,
then we get the following result which is of particular importance.

\begin{cor}
\label{cor:main}
  Suppose the function $\m$ satisfies \eqref{eqn:limsup m <= limsup d} and
  $d(x_n,x_{n+1})\to0$. Then each of the following conditions implies that
  $\set{x_n}$ is Cauchy.
  \begin{enumerate}[\upshape(i)]
    \item \label{item:cor:main:nu=0}
    for every $\e>0$ and for any two subsequences $\{x_{p_n}\}$ and $\{x_{q_n}\}$,
    if \newline
    $\limsup \m(x_{p_n},x_{q_n}) \leq \e$, then, for some $N$,
    \begin{equation*}
      d(x_{p_n},x_{q_n}) \leq \e, \qquad (n\geq N).
    \end{equation*}

    \item \label{item:cor:main:nu=1}
    for every $\e>0$ and for any two subsequences $\{x_{p_n}\}$ and $\{x_{q_n}\}$,
    if \newline
    $\limsup \m(x_{p_n},x_{q_n}) \leq \e$, then, for some $N$,
    \begin{equation*}
      d(x_{p_n+1},x_{q_n+1}) \leq \e, \qquad (n\geq N).
    \end{equation*}

    \item \label{item:cor:main:nu}
    for every $\e>0$, there exists $\nu\in\N$
    such that, for any two subsequences
    $\{x_{p_n}\}$ and $\{x_{q_n}\}$, if\/ $\limsup \m(x_{p_n},x_{q_n})\leq \e$,
    then, for some $N\in\N$,
    \begin{equation*}
      d(x_{p_n+\nu},x_{q_n+\nu}) \leq \e, \qquad (n\geq N).
    \end{equation*}

  \end{enumerate}
\end{cor}

\bigskip

\section{Fixed Point Theorems}
\label{sec:fixed-point-theorems}

In this section, using Theorem \ref{thm:main}, we present our fixed point theorem.

\begin{lem}
\label{lem:equiv-conditions-m-contractive-sequence}
 If $\set{x_n}$ is a sequence in a metric space $X$ and
 $\m$ is a nonnegative function on $X\times X$, then
 the following statements are equivalent:
  \begin{enumerate}[\upshape(i)]
    \item \label{item:it-is-m-contractive-sequence}
    for every $\e>0$, there exists $\de>0$ and $N\in\N_0$ such that
    \begin{equation}\label{eqn:m-contractive-sequence}
      \forall p,q\geq N, \quad
      \m(x_p,x_q) < \e+\de \Longrightarrow d(x_{p+1},x_{q+1}) \leq \e.
    \end{equation}

    \item \label{item:m-contractive-sequence-in-term-of-pnqn}
    for every $\e>0$, and for any two subsequences $\{x_{p_n}\}$ and $\{x_{q_n}\}$,
    if \newline $\limsup \m(x_{p_n},x_{q_n})\leq \e$ then,
    for some $N$,
    \[
      d(x_{p_n+1},x_{q_n+1}) \leq \e, \qquad (n\geq N).
    \]
  \end{enumerate}
\end{lem}

\begin{proof}
 $\eqref{item:it-is-m-contractive-sequence}
 \Rightarrow \eqref{item:m-contractive-sequence-in-term-of-pnqn}$:
 Let $\e>0$ and assume, for subsequences $\{x_{p_n}\}$ and $\{x_{q_n}\}$,
 we have $\limsup \m(x_{p_n},x_{q_n})\leq \e$. There exists,
 by \eqref{item:it-is-m-contractive-sequence},
 some $\de>0$ and $N\in\N_0$ such that \eqref{eqn:m-contractive-sequence}
 holds. Take $N_1\in\N_0$ such that $\m(x_{p_n},x_{q_n}) < \e+\de$
 for $n\geq N_1$. Therefore, we have $d(x_{p_n+1},x_{q_n+1}) \leq \e$,
 for $n>\max\set{N,N_1}$.

 $\eqref{item:m-contractive-sequence-in-term-of-pnqn}
 \Rightarrow \eqref{item:it-is-m-contractive-sequence}$:
 Assume, to get a contradiction, that \eqref{item:it-is-m-contractive-sequence}
 fails to hold. Then there exist $\e>0$ and subsequences $\{x_{p_n}\}$ and $\{x_{q_n}\}$
 such that
 \[
   \m(x_{p_n},x_{q_n}) < \e+\frac1n \quad \text{and} \quad
   \e < d(x_{p_n+1},x_{q_n+1}).
 \]
 This contradicts \eqref{item:m-contractive-sequence-in-term-of-pnqn}
 because $\limsup \m(x_{p_n},x_{q_n}) \leq \e$.
\end{proof}

\begin{dfn}
\label{dfn:m-contractive-sequence}
 Let $X$ be a metric space and let $\m$ be a nonnegative function on
 $X\times X$. A sequence $\set{x_n}$ in $X$ is said to be
 \emph{$\m$-contractive} if it satisfies one (and hence all) of the conditions in
 Lemma \ref{lem:equiv-conditions-m-contractive-sequence}.
\end{dfn}

The following is a direct consequence of Corollary \ref{cor:main} and
the above lemma.

\begin{thm}
\label{thm:m-contractive-sequences-are-Cauchy}
  Let $X$ be a metric space, $\{x_n\}$ be a sequence in $X$, and
  $\m$ be a nonnegative function on $X\times X$ satisfying \eqref{eqn:limsup m <= limsup d}.
  If $\{x_n\}$ is $\m$-contractive and $d(x_n,x_{n+1})\to0$, then $\{x_n\}$ is Cauchy.
\end{thm}

\begin{cor}
\label{cor:m-contractive-orbits}
  Let $T$ be a self-map of a metric space $X$, and $\m$ be a nonnegative function on $X\times X$.
  Suppose there exists a point $x\in X$ such that
  \begin{enumerate}[\upshape(i)]
    \item for any $\e>0$, there exist $\de>0$ and $N\in\N_0$ such that
    \begin{equation}\label{eqn:m-contractive-orbits}
      \forall p,q\geq N, \quad
       \m(T^px, T^qx) < \de+\e \Longrightarrow d(T^{p+1} x, T^{q+1}x)\leq \e,
    \end{equation}

    \item condition \eqref{eqn:limsup m <= limsup d} holds for any two subsequences
    $\{x_{p_n}\}$ and $\{x_{q_n}\}$ of\/ $\{T^nx\}$.
  \end{enumerate}
  If $d(T^nx,T^{n+1}x)\to0$, then $\set{T^nx}$ is a Cauchy sequence.
\end{cor}

The requirement that $d(x_n,x_{n+1})\to0$ is essential in Lemma \ref{lem:technical-lemma}
and its subsequent results. It can, however, be replaced by other conditions.

\begin{prop}
\label{prop:contractive-sequences}
  Let $\{x_n\}$ be a sequence in a metric space $X$ such that
  \begin{equation}\label{eqn:contractive-sequences}
    \begin{cases}
      d(x_{n+1},x_{n+2}) \leq d(x_n,x_{n+1}), & (n\in\N_0), \\
      d(x_{n+1},x_{n+2}) < d(x_n,x_{n+1}), & (\text{if $x_n\neq x_{n+1}$}).
    \end{cases}
  \end{equation}
  If $\m$ satisfies \eqref{eqn:limsup m <= limsup d} and $\{x_n\}$ is $\m$-contractive,
  then $d(x_n,x_{n+1})\to0$ and, hence, $\{x_n\}$ is Cauchy.
\end{prop}

For instance, if $T:X\to X$ is contractive 
then \eqref{eqn:contractive-sequences} holds for $x_n=T^nx$.

\begin{proof}
 If $x_m=x_{m+1}$, for some $m$, then $x_n=x_{n+1}$ for all $n\geq m$,
 and there is nothing to prove. Assume that $x_n\neq x_{n+1}$ for all $n$.
 Then $d(x_{n+1},x_{n+2})< d(x_n,x_{n+1})$, for every $n$, and thus
 $d(x_n,x_{n+1}) \downarrow \e$, for some $\e\geq0$.
 If $\e>0$, take $p_n=n$ and $q_n=n+1$ and we have, by \eqref{eqn:limsup m <= limsup d},
 \[
   \limsup_{n\to\infty} \m(x_n,x_{n+1})
    \leq \limsup_{n\to\infty} d(x_n,x_{n+1})=\e.
 \]
 Therefore, $d(x_{n+1},x_{n+2}) \leq \e$ for $n$ large.
 This is a contradiction since $\e<d(x_n,x_{n+1})$ for all $n$.
 So $\e=0$ and $d(x_n,x_{n+1})\to0$.
\end{proof}

We are now in a position to state and prove our fixed point theorem.

\begin{thm}
\label{thm:fixed-point-theorem}
  Let $T$ be an asymptotically regular self-map of a metric space $X$, and
  $\m$ be a nonnegative function on $X\times X$.
  Suppose
  \begin{enumerate}[\upshape(i)]
    \item \label{item:m-contractive-in-pre-fixed-point-theorem}
    for any $\e>0$, there exist $\de>0$ and $N\in\N_0$ such that
    \begin{equation}\label{eqn:m-contractive-in-fixed-point-theorem}
      \forall x,y\in X, \quad
       \m(T^Nx, T^Ny) < \de+\e \Longrightarrow d(T^{N+1}x,T^{N+1}y)\leq \e,
    \end{equation}

    \item \label{item:m-satisfies-in-pre-fixed-point-theorem}
    for every $x\in X$, condition \eqref{eqn:limsup m <= limsup d}
    holds for any two subsequences $\{x_{p_n}\}$ and $\{x_{q_n}\}$
    of the sequence $\{T^nx\}$.
  \end{enumerate}
  Then the Picard iterates of $T$ are Cauchy sequences.
  Moreover, if $X$ is complete, $T$ is continuous, and $d(T^nx,T^ny)\to0$,
  for all $x,y\in X$, then the Picard iterates converge to a unique fixed point
  of $T$.
\end{thm}

\begin{proof}
  All conditions in Corollary \ref{cor:m-contractive-orbits} are satisfied by every
  point $x$ in $X$. Hence $\{T^nx\}$ is Cauchy, for every $x$.
  If $X$ is complete, there is $z\in X$ such that $T^nx\to z$. If $T$ is continuous,
  then $Tz=z$. If $d(T^nx,T^ny)\to0$, for every $x,y\in X$, then $T$ has
  at most one fixed point.
\end{proof}

We remark that, by Corollary \ref{cor:m-contractive-orbits}, it is enough to impose
condition \eqref{eqn:m-contractive-in-fixed-point-theorem} on some orbit $\set{T^nx}$ as in
\eqref{eqn:m-contractive-orbits} to conclude that $\set{T^nx}$ is Cauchy.

\begin{example}
  Let $T$ be a self-map of $X$, and consider the following functions:
  \begin{align*}
    \m_1(x,y) & = \max\set{d(x,Tx),d(y,Ty)}, \tag{Bianchini \cite{Bianchini-1972}}\\
    \m_2(x,y) & = [d(x,Ty)+d(y,Tx)]/2, \tag{Chatterjea \cite{Chatterjea-1972}} \\
    \m_3(x,y) & = \max\set{d(x,y),d(x,Tx),d(y,Ty)},
                    \tag{Maiti and Pal \cite{Maiti-Pal-1978}} \\
    \m_4(x,y) & = \max\set{d(x,y),d(x,Tx),d(y,Ty),d(x,Ty),d(y,Tx)},
                    \tag{Ciric \cite{Ciric-1974}} \\
    \m_5(x,y) & = \max\set{d(x,y),d(x,Tx),d(y,Ty),[d(x,Ty)+d(y,Tx)]/2},
                    \tag{Jachymski \cite{Jachymski-1995}}\\
    \m_6(x,y) & = d(x,y) + \ga[d(x,Tx)+ d(y,Ty)],\ \text{where $\ga\geq0$ is fixed}.
                    \tag{Proinov \cite{Proinov-2006}}
  \end{align*}

  Choose a point $x\in X$ and set $x_n=T^nx$, $n\in\N_0$.
  If $d(x_n,x_{n+1})\to0$, then,
  for any two subsequences $\{x_{p_n}\}$ and $\{x_{q_n}\}$, we have
  \begin{equation*}
    \begin{cases}
      \limsup\limits_{n\to\infty}\m_1(x_{p_n},x_{q_n})=0, & \\[1ex]
      \limsup\limits_{n\to\infty}\m_2(x_{p_n},x_{q_n})
        \leq \limsup\limits_{n\to\infty} d(x_{p_n},x_{q_n}), & \\[1ex]
      \limsup\limits_{n\to\infty} \m_i(x_{p_n},x_{q_n})
     = \limsup\limits_{n\to\infty} d(x_{p_n},x_{q_n}),      & \text{for $3\leq i \leq 6$}.
    \end{cases}
  \end{equation*}
\end{example}

We now state and prove a generalization of
Proinov's Theorem \ref{thm:Proinov}. First, a couple of notations:
For a subset $E$ of a metric space $X$, denote by
$\diam E$ the diameter of $E$. If $T$ is a self-map
of $X$ and $x\in X$, for every positive integer $s\in \N$,
let $\cO_s(x)=\set{T^nx: 0 \leq n \leq s}$.
For positive integers $s,t\in\N$ and real numbers $\al,\be\in[0,\infty)$, define
a function $\m$ on $X\times X$ as follows:
\begin{equation}\label{eqn:m(x,y)-in-generalized-Proinov}
  \m(x,y) = d(x,y) +
  \al \diam \cO_s(x) + \be \diam \cO_t(y).
\end{equation}

\begin{thm}
\label{thm:Generalized-Proinov}
  Let $X$ be a complete metric space, and $T$ be a continuous and
  asymptotically regular self-map of $X$. Define $\m$ by
  \eqref{eqn:m(x,y)-in-generalized-Proinov}, and suppose
  \begin{enumerate}[\upshape(i)]
    \item \label{item:d(Tx,Ty)<m(x,y)-in-generalized-Proinov}
    $d(Tx,Ty)<\m(x,y)$, for every $x,y\in X$ with $x\neq y$,

    \item \label{item:m-contractive-in-generalized-Proinov}
    for any $\e>0$, there exist $\de>0$ and $N\in\N_0$ such that
    \begin{equation}\label{eqn:m-contractive-in-generalized-Proinov}
      \forall x,y\in X, \quad
       \m(T^Nx, T^Ny) < \de+\e \Longrightarrow d(T^{N+1}x,T^{N+1}y)\leq \e.
    \end{equation}
  \end{enumerate}
  Then $T$ has a unique fixed point $z$, and the Picard iterates of $T$
  converge to $z$.
\end{thm}

\begin{proof}
  First, we prove that $T$ has at most one fixed point. If $Tx=x$ and
  $Ty=y$ then $\cO_s(x)=\{x\}$ and $\cO_t(y)=\{y\}$ and thus
  \[
    \m(x,y)=d(x,y)=d(Tx,Ty).
  \]
  Condition \eqref{item:d(Tx,Ty)<m(x,y)-in-generalized-Proinov}
  then implies that $x=y$.

  Now, choose $x\in X$ and set $x_n=T^nx$, $n\in\N_0$.
  Then $d(x_n,x_{n+1})\to0$ since $T$ is asymptotically regular. It is easy to see that,
  for any two subsequences $\{x_{p_n}\}$ and $\{x_{q_n}\}$, we have
  \[
    \m(x_{p_n},x_{q_n})
     \leq d(x_{p_n},x_{q_n})
      + \al \sum_{i=0}^{s-1} d(x_{p_n+i},x_{p_n+i+1})
      + \be \sum_{j=0}^{t-1} d(x_{q_n+j},x_{q_n+j+1}).
  \]
  Since $d(x_n,x_{n+1})\to0$, we see that \eqref{eqn:limsup m <= limsup d}
  holds. Hence, by Theorem \ref{thm:fixed-point-theorem}, the sequence
  $\{T^nx\}$ is Cauchy and, since $X$ is complete,
  it converges to some point $z\in X$.
  Since $T$ is continuous, we have $Tz=z$.
\end{proof}

We next give an example to show that Theorem \ref{thm:Generalized-Proinov}
strictly extends Proinov's Theorem \ref{thm:Proinov}.

\begin{example}
  Take $a_i=i$, for $0\leq i < 4$, let $r_0=0$, and $r_n=1/n$ for
  $n\geq1$, and set $x_{4n+i}=a_i+r_n$. Let
  \[
    X=\set{x_{4n+i}:0\leq i<4,\,n\geq0}.
  \]
  Then, equipped with the Euclidean metric, $X$ is a complete metric space. Define
  a mapping $T:X\to X$ by $T(x_\ell)=x_{2\ell}$. Define $\m(x,y)$ by setting $s=t=1$
  and $\al=\be=1$ in \eqref{eqn:m(x,y)-in-generalized-Proinov}, that is,
  \[
    \m(x,y)=d(x,y)+d(x,Tx)+d(y,Ty).
  \]
  (Note that $\m$ is also obtained from \eqref{eqn:Proinov-m(x,y)} by setting $\ga=1$.)

  First, we show that $T$ satisfies all conditions in Theorem \ref{thm:Generalized-Proinov}.
  Clearly, $T$ is continuous and, for every $x_\ell,x_\nu\in X$, we have
  \[
    d(T^{n+2}x_\ell,T^{n+2}x_\nu)
      = d(T^n x_{4\ell},T^n x_{4\nu})
      = \Bigl|\frac1{2^{n+1}\ell}-\frac1{2^{n+1}\nu}\Bigr|\to0.
  \]

  It is a matter of calculation to see that $d(Tx,Ty)<\m(x,y)$, for all $x,y\in X$.
  The following shows that $T$ satisfies \eqref{eqn:m-contractive-in-generalized-Proinov}
  with $N=2$:
  \begin{align*}
    d(T^3x_\ell,T^3x_\nu)
    & =|x_{8\ell}-x_{8\nu}|=\Bigl|\frac1{8\ell}-\frac1{8\nu}\Bigr|
      = \frac12\Bigl|\frac1{4\ell}-\frac1{4\nu}\Bigr|\\
    & = \frac12 |x_{4\ell}-x_{4\nu}| = \frac12 d(T^2x_\ell,T^2x_\nu).
  \end{align*}

  Next, we show that the following condition (in Proinov's theorem) is violated:
  \begin{itemize}
    \item for every $\e>0$, there exist $\de>0$ such that
    \begin{equation*}
      \forall x,y\in X, \quad
       \m(x, y) < \de+\e \Longrightarrow d(Tx,Ty)\leq \e.
    \end{equation*}
  \end{itemize}
  Take $\e=a_2-a_0$. Let $u_\ell=x_{4\ell}$ and $v_\ell=x_{4(\ell+1)+1}$.
  Then
  \begin{align*}
    d(Tu_\ell,Tv_\ell)
     & =|x_{4(2\ell)}-x_{4(2\ell+2)+2}| = a_2-a_0 + r_{2\ell}-r_{2\ell+2}>\e,
  \intertext{and}
    \m(u_\ell,v_\ell)
     & = |u_\ell-v_\ell|+|u_\ell-Tu_\ell|+|v_\ell-Tv_\ell| \\
     & = |x_{4\ell}-x_{4(\ell+1)+1}|+|x_{4\ell}-x_{4(2\ell)}|+|x_{4(\ell+1)+1}-x_{4(2\ell+2)+2}|\\
     & = (a_1-a_0+r_\ell-r_{\ell+1})+(r_\ell-r_{2\ell})+(a_2-a_1+r_{\ell+1}-r_{2\ell+2})\\
     & = a_2-a_0 + 2r_\ell-r_{2\ell}-r_{2\ell+2}.
  \end{align*}
  If $\de_\ell=2r_\ell-r_{2\ell}-r_{2\ell+2}$, we have
  \[
    \e < d(Tu_\ell,Tv_\ell) < \m(u_\ell,v_\ell) \leq \e+\de_\ell.
  \]
  Since $\de_\ell\to0$ as $\ell\to\infty$, we see that $T$
  does not satisfy \eqref{eqn:m-contractive-in-generalized-Proinov}
  for $\e=a_2-a_0$.
\end{example}

We conclude this section by showing that the following theorem of
Geraghty \cite{Geraghty-73} is a special case of our fixe point theorem.

\begin{thm}[Geraghty \cite{Geraghty-73}]
\label{thm:Geraghty}
  Let $T$ be a contractive self-map of a complete metric space $X$,
  let $x\in X$, and set $x_n=T^n x$, $n\in\N_0$. Then $\set{x_n}$
  converges to a unique fixed point of $T$ if and only if,
  for any two subsequences $\set{x_{p_n}}$ and $\set{x_{q_n}}$, with
  $x_{p_n}\neq x_{q_n}$, condition
  \begin{equation*}
    \lim_{n\to\infty}\frac{d(Tx_{p_n},Tx_{q_n})}{d(x_{p_n},x_{q_n})}=1,
  \end{equation*}
  implies $d(x_{p_n},x_{q_n})\to0$.
\end{thm}

\begin{prop}
\label{prop:G-contractive-property-is-a-general-one}
  Let $\set{x_n}$ be a sequence in a metric space $X$ such that
  $d(x_{p+1},x_{q+1}) \leq d(x_p,x_q)$, for all $p,q\in \N_0$.
  Then the following statements are equivalent:
  \begin{enumerate}[\upshape(i)]
    \item \label{item:it-is-a-Geraghty-sequence}
    for any two subsequences $\set{x_{p_n}}$ and $\set{x_{q_n}}$, with
    $x_{p_n}\neq x_{q_n}$, condition
    \begin{equation}\label{eqn:lim=1-in-Geraghty-equiv}
      \lim_{n\to\infty}\frac{d(x_{p_n+1},x_{q_n+1})}{d(x_{p_n},x_{q_n})}=1,
    \end{equation}
    implies $d(x_{p_n},x_{q_n})\to0$.

    \item \label{item:Geraghty-sequence-in-term-of-pnqn}
    for every $\e>0$, for any two subsequences $\{x_{p_n}\}$ and $\{x_{q_n}\}$,
    if\newline  $\limsup d(x_{p_n},x_{q_n})\leq \e$ then
    \begin{equation}\label{eqn:Geraghty-sequence-in-term-of-pnqn}
      \limsup_{n\to\infty}d(x_{p_n+1},x_{q_n+1}) < \e.
    \end{equation}

    \item \label{item:Geraghty-sequence-in-e-de-eta}
    for every $\e>0$, there exist $\de>0$, $\eta\in(0,\e)$, and $N\in\N_0$, such that
    \begin{equation}\label{eqn:Geraghty-sequence-in-e-de-eta}
      \forall p,q \geq N,\quad
      d(x_p,x_q) < \e + \de \Longrightarrow
      d(x_{p+1},x_{q+1}) \leq \eta.
    \end{equation}
  \end{enumerate}
\end{prop}

\begin{proof}
  $\eqref{item:it-is-a-Geraghty-sequence}
  \Rightarrow \eqref{item:Geraghty-sequence-in-term-of-pnqn}$:
  Assume that, for two subsequences $\{x_{p_n}\}$ and $\{x_{q_n}\}$, we have
  \[
    \limsup_{n\to\infty} d(x_{p_n+1},x_{q_n+1})
    =\limsup_{n\to\infty} d(x_{p_n},x_{q_n})>0.
  \]
  Since $d(x_{p_n+1},x_{q_n+1})\leq d(x_{p_n},x_{q_n})$ for all $n$,
  by passing through subsequences, if necessary, we can assume that
  \[
    \lim_{n\to\infty} d(x_{p_n+1},x_{q_n+1})
    =\lim_{n\to\infty} d(x_{p_n},x_{q_n})>0.
  \]
  Therefore, we get \eqref{eqn:lim=1-in-Geraghty-equiv}. Since
  $d(x_{p_n},x_{q_n})$ does not converge to $0$, we conclude that the sequence
  does not satisfy \eqref{item:it-is-a-Geraghty-sequence}.

  $\eqref{item:Geraghty-sequence-in-term-of-pnqn}
  \Rightarrow \eqref{item:Geraghty-sequence-in-e-de-eta}$:
  Assume there is $\e>0$ such that, for every $n\in\N_0$, there exist $p_n$ and $q_n$ with
  $q_n>p_n\geq n$ such that
  \[
    d(x_{p_n},x_{q_n}) < \e + \frac1n
    \quad \text{and} \quad
    \e-\frac1n \leq d(x_{p_n+1},x_{q_n+1}).
  \]
  Then
  \[
    \e \leq \limsup_{n\to\infty} d(x_{p_n+1},x_{q_n+1}) \leq \limsup_{n\to\infty} d(x_{p_n},x_{q_n}) \leq \e.
  \]

  $\eqref{item:Geraghty-sequence-in-e-de-eta}
  \Rightarrow \eqref{item:it-is-a-Geraghty-sequence}$:
  Assume that, for two subsequences $\set{x_{p_n}}$ and $\set{x_{q_n}}$, condition \eqref{eqn:lim=1-in-Geraghty-equiv}
  holds and also
  \[
    \e=\limsup_{n\to\infty} d(x_{p_n},x_{q_n})>0.
  \]
  For this $\e$, by \eqref{item:Geraghty-sequence-in-e-de-eta}, there
  exist $\de>0$, $\eta\in(0,\e)$, and $N\in\N_0$, such that
  \eqref{eqn:Geraghty-sequence-in-e-de-eta} holds true.
  There is $N_1>N$ such that, for $n\geq N_1$, we have $d(x_{p_n},x_{q_n}) < \e + \de$
  and thus $d(x_{p_n+1},x_{q_n+1})\leq \eta$. On the other hand \eqref{eqn:lim=1-in-Geraghty-equiv}
  implies that, for every $r<1$, there is $N_2>N_1$ such that, for $n\geq N_2$,
  \[
    r d(x_{p_n},x_{q_n}) \leq d(x_{p_n+1},x_{q_n+1})\leq \eta.
  \]
  If $n\to\infty$ we get $r\e \leq \eta$. If $r\to 1$, we get $\e \leq \eta$
  which is absurd.
\end{proof}

Now, to see why Theorem \ref{thm:Geraghty} follows from the results in this section,
take a point $x\in X$ and set $x_n=T^nx$, $n\in\N_0$. If $\set{x_n}$ satisfies
the condition in Theorem \ref{thm:Geraghty}, then, by Proposition
\ref{prop:G-contractive-property-is-a-general-one}, the sequence $\set{x_n}$ is
$d$-contractive (in the sense of Definition \ref{dfn:m-contractive-sequence}).
On the other hand, $T$ being contractive implies that $\set{x_n}$ satisfies
\eqref{eqn:contractive-sequences}. Hence, by Proposition \ref{prop:contractive-sequences},
we have $d(x_n,x_{n+1})\to0$. Theorem \ref{thm:m-contractive-sequences-are-Cauchy} now
implies that $\set{x_n}$ is a Cauchy sequence.

\bigskip

\section{Asymptotic Contractions of Meir-Keeler Type}
\label{sec:asymptotic-contractions}

In 2003, Kirk \cite{Kirk-2003} introduced the notion of asymptotic contraction
on a metric space, and proved a fixed-point theorem for such contractions
(see also \cite{Arandelovic-2005}). In 2006, Suzuki \cite{Suzuki-AC-2006}
introduced the notion of asymptotic contraction
of Meir-Keeler type, and proved a fixed-point theorem for such contractions, which is
a generalization of both Meir and Keeler's theorem \cite{Meir-Keeler-1969}
and Kirk's theorem \cite{Kirk-2003}. A year later, Suzuki \cite{Suzuki-AC-2007}
introduced the following notion of asymptotic contractions which is, in some sense, the final
definition of asymptotic contractions (see \cite[Theorem 6]{Suzuki-AC-2007}).

\begin{dfn}[Suzuki \cite{Suzuki-AC-2007}]
\label{dfn:ACF}
  A mapping $T$ on a metric space $X$ is said to be an
  \emph{asymptotic contraction of the final type} if
  \begin{enumerate}[\upshape(i)]
    \item $d(T^nx,T^ny)\to0$, for all $x,y\in X$,

    \item for every $x\in X$ and $\e>0$, there exist $\de>0$ and $\nu\in\N$ such that
    \begin{equation*}
      \forall p,q\in\N, \quad
      \e<d(T^p x,T^q x) < \e+\de \Longrightarrow
      d(T^{p+\nu} x,T^{q+\nu} x) \leq \e.
    \end{equation*}

  \end{enumerate}
\end{dfn}

Then they proved the following result.

\begin{thm}[Suzuki \cite{Suzuki-AC-2007}]
\label{thm:Suzuki-AC-Cauchy}
  Let $X$ be a metric space and let $T$ be an asymptotic contraction of the final type
  on $X$. Then $\{T^nx\}$, for every $x$, is a Cauchy sequence.
\end{thm}

We present a short proof of the above theorem using the results in section \ref{sec:technical-lemma}.
But, first, let us make the following definition.

\begin{dfn}
\label{dfn:abtahi-ACF}
  Let $T$ be a mapping on a metric space $X$, and $\m$ a nonnegative function on $X\times X$.
  We call $T$ an \emph{asymptotic $\m$-contraction} if
  \begin{enumerate}[\upshape(i)]
    \item \label{item:d(Tnx,Tny)->0}
    $d(T^nx,T^ny)\to0$, for all $x,y\in X$,

    \item for every $x\in X$ and $\e>0$, there exist $\de>0$, $\nu\in\N$, and $N\in\N_0$
    such that
    \begin{equation*}
      \forall p,q\geq N, \quad
      \m(T^p x,T^q x) < \e+\de \Longrightarrow
      d(T^{p+\nu} x,T^{q+\nu} x) \leq \e.
    \end{equation*}
  \end{enumerate}
\end{dfn}

\begin{thm}
\label{thm:abtahi-AC-Cauchy}
  Let $T$ be an asymptotic $\m$-contraction on $X$. If $\m$ satisfies
  \eqref{eqn:limsup m <= limsup d}, for some $x\in X$, then $\{T^nx\}$ is a Cauchy sequence.
\end{thm}

\begin{proof}
  Let $x_n=T^nx$, $n\in\N_0$. Then the following are equivalent (the proof is similar to that of Lemma
  \ref{lem:equiv-conditions-m-contractive-sequence} and hence is omitted).
  \begin{enumerate}[\upshape(i)]
    \item for every $\e>0$, there exist $\nu\in\N$, $\de>0$ and $N\in\N_0$ such that
     \[
       \forall\, p,q \geq N,\quad
       \m(x_p,x_q) < \e+\de \Longrightarrow
       d(x_{p+\nu},x_{q+\nu}) \leq \e,
     \]

    \item for every $\e>0$, there exists $\nu\in\N$ such that,
    for any two subsequences $\{x_{p_n}\}$ and $\{x_{q_n}\}$,
    if\/ $\limsup \m(x_{p_n},x_{q_n})\leq \e$, then, for some $N$,
    \begin{equation*}
      d(x_{p_n+\nu},x_{q_n+\nu}) \leq \e, \qquad (n\geq N).
    \end{equation*}
  \end{enumerate}

  Condition \eqref{item:d(Tnx,Tny)->0} in Definition \ref{dfn:abtahi-ACF}
  implies that $d(x_n,x_{n+1})\to0$. Now part \eqref{item:cor:main:nu} of
  Corollary \ref{cor:main} shows that $\{x_n\}$ is Cauchy.
\end{proof}

At first look, because of replacing the distance function $d(x,y)$ with
a more general function $\m(x,y)$, such as the one defined by \eqref{eqn:m(x,y)-in-generalized-Proinov},
it may seem that Definition \ref{dfn:abtahi-ACF} and
Theorem \ref{thm:abtahi-AC-Cauchy} are extensions of Suzuki's Definition \ref{dfn:ACF} and
Theorem \ref{thm:Suzuki-AC-Cauchy}, respectively. However, we have the following result which
confirms that Definition \ref{dfn:ACF} is the final definition of asymptotic contractions.

\begin{thm}
\label{thm:final}
  Let $X$ be a complete metric space. If $T$ is an asymptotic $\m$-contraction,
  for some $\m$ satisfying \eqref{eqn:limsup m <= limsup d} for all $x$, then $T$ is
  an asymptotic contraction of the final type.
\end{thm}

\begin{proof}
  For every $x\in X$, by Theorem \ref{thm:abtahi-AC-Cauchy}, the sequence $\set{T^nx}$ is Cauchy,
  and, since $X$ is complete, there is $z\in X$ such that $T^nx\to z$.  Since $d(T^nx,T^ny)\to0$,
  for all $x,y\in X$, the point $z$ is unique. Now, Theorem 8 in \cite{Suzuki-AC-2007} shows
  that $T$ is an asymptotic contraction of the final type.
\end{proof}

\bigskip

\section*{Acknowledgment}

The author expresses his sincere gratitude to the anonymous referee for his/her careful reading
and suggestions that improved the presentation of this paper.

\bigskip

{\it Received: ; Accepted: }


\begin{thebibliography}{99}

\bibitem{Arandelovic-2005}
  I.D. Arandelovic,
  \textit{On a fixed point theorem of Kirk},
  J. Math. Anal. Appl. \textbf{301} (2005) 384--385.

\bibitem{Banach}
  S. Banach,
  \textit{Sur les operations dans les ensembles abstraits et leur application aux equations
  integrales},
  Fund. Math., \textbf{3} (1922), 133--181.

\bibitem{Bianchini-1972}
  R.M.T. Bianchini,
  \textit{Su un problema di S. Reich riguardante la teoria dei punti fissi},
  Boll. Un. Mat. Ital. \textbf{5} (1972), 103--108.

\bibitem{Boyd-Wong-1969}
  D.W. Boyd, J.S. Wong,
  \textit{On nonlinear contractions},
  Proc. Amer. Math. Soc. \textbf{20} (1969), 458--469.

\bibitem{Browder-1966}
  F.E. Browder, W.V. Petryshyn,
  \textit{The solution by iteration of nonlinear functional equations in Banach spaces},
  Bull. Amer. Math. Soc. \textbf{72} (1966), 571--575.

\bibitem{Caristi-1976}
  J. Caristi,
  \textit{Fixed point theorems for mappings satisfying inwardness conditions},
  Trans. Amer. Math. Soc., \textbf{215} (1976), 241--251.

\bibitem{Cho-Murthy-Jungck-2000}
  Y.J. Cho, P.P. Murthy, G. Jungck,
  \textit{A theorem of Meer-Keeler type revisited},
  Int. J. Math. Math. Sci. \textbf{23} (2000) 507--511.

\bibitem{Chatterjea-1972}
  S.K. Chatterjea,
  \textit{Fixed-point theorems},
  C. R. Acad. Bulgare Sci. \textbf{25} (1972), 727--730.

\bibitem{Ciric-1971}
  Lj. B. \'Ciri\'c,
  \textit{Generalized contractions and fixed-point theorems},
  Publ. Inst. Math. (Beograd) (N.S.) \textbf{12} (26) (1971), 19--20.

\bibitem{Ciric-1974}
  Lj. B. \'Ciri\'c,
  \textit{A generalization of Banach's contraction principle},
  Proc. Amer. Math. Soc., \textbf{45} (1974), 267--273.

\bibitem{Ciric-1981}
  Lj. B. \'Ciri\'c,
  \textit{A new fixed-point theorem for contractive mappings},
  Publ. Inst. Math. (N.S) \textbf{30} (44) (1981), 25--27.

\bibitem{Edelstein-62}
  M. Edelstein,
  \textit{On fixed and periodic points under contractive mappings},
  J. London Math. Soc. \textbf{37} (1962), 74--79.

\bibitem{Ekeland-1974}
  I. Ekeland,
  \textit{On the variational principle},
  J. Math. Anal. Appl., \textbf{47} (1974), 324--353.

\bibitem{Geraghty-73}
  M.A. Geraghty,
  \textit{On contractive mappings},
  Proc. Amer. Math. Soc., \textbf{40} (1973), 604--608.

\bibitem{Hardy-Rogers-1973}
  G.E. Hardy, T.D. Rogers,
  \textit{A generalization of a fixed point theorem of Reich},
  Canad. Math. Bull. \textbf{16} (1973), 201--206.

\bibitem{Jachymski-1995}
  J. Jachymski,
  \textit{Equivalent conditions and the Meir-Keeler type theorems},
  J. Math. Anal. Appl. \textbf{194} (1995), 293--303.

\bibitem{Kannan-1969}
  R. Kannan,
  \textit{Some rsults on fixed points. II},
  Amer. Math. Monthly \textbf{76} (1969), 405--408.

\bibitem{Kirk-2003}
  W.A. Kirk,
  \textit{Fixed points of asymptotic contractions},
  J. Math. Anal. Appl., \textbf{277} (2003), 645--650.

\bibitem{Kuczma}
  M. Kuczma, B. Choczewski, R. Ger,
  \textit{Iterative functional equations, Encyclopedia of Mathematics and Applications},
  vol. 32, Cambridge University Press, Cambridge, 1990.

\bibitem{Leader-1983}
  S. Leader,
  \textit{Equivalent Cauchy sequences and contractive fixed points in metric spaces},
  Studia Math. \textbf{76} (1983) 63--67.

\bibitem{Maiti-Pal-1978}
  M. Maiti, T.K. Pal,
  \textit{Generalization of two fixedpoint theorems},
  Bull. Calcutta Math. Soc. \textbf{70} (1978), 57--61.

\bibitem{Matkowski-1975}
  J. Matkowski,
  \textit{Integrable solutions of functional equations},
  Diss. Math. \textbf{127} (1975) Warsaw.

\bibitem{Matkowski-1980}
  J. Matkowski,
  \textit{Fixed point theorems for contractive mappings in metric spaces},
  Cas. Pest. Mat. \textbf{105} (1980), 341--344.

\bibitem{Meir-Keeler-1969}
  A. Meir, E. Keeler,
  \textit{A theorem on contraction mappings},
  J. Math. Anal. Appl., \textbf{28} (1969), 326--329.

\bibitem{Park-Rhoades-1981}
  S. Park, B.E. Rhoades,
  \textit{Meir-Keeler type contractive conditions},
  Math. Japon. \textbf{26} (1) (1981), 13--20.

\bibitem{Proinov-2006}
  Petko D. Proinov,
  \textit{Fixed point theorems in metric spaces},
  Nonlinear Anal. \textit{64} (2006), 546--557.

\bibitem{Rao-Rao-1985}
  I.H.N. Rao, K.P.R. Rao,
  \textit{Generalizations of fixed point theorems of Meir and Keeler type},
  Indian J. Pure Appl. Math. \textbf{16} (1) (1985), 1249--1262.

\bibitem{Rhoades-1977}
  B.E. Rhoades,
  \textit{A comparison of various definitions of contractive mappings},
  Trans. Amer. Math. Soc. \textbf{226} (1977), 257--290.

\bibitem{Sehgal-1972}
  V.M. Sehgal,
  \textit{On fixed and periodic points for a class of mappings},
  J. London Math. Soc. (2) \textbf{5} (1972), 571--576.

\bibitem{Singh-1969}
  S.P. Singh,
  \textit{Some results on fixed point theorems},
  Yokahama Math. J. \textbf{17} (1969), 61--64.

\bibitem{Subrahmanyam-1974}
  P.V. Subrahmanyam,
  \textit{Remarks on some fixed point theorems related to Banach's contraction principle},
  J. Math. Phys. Sci., \textit{8} (1974), 445--457.

\bibitem{Suzuki-AC-2006}
  T. Suzuki,
  \textit{Fixed-point theorem for asymptotic contractions of
  Meir–Keeler type in complete metric spaces},
  Nonlinear Analysis \textbf{64} (2006) 971--978.

\bibitem{Suzuki-AC-2007}
  T. Suzuki,
  \textit{A definitive result on asymptotic contractions},
  J. Math. Anal. Appl. \textbf{335} (2007) 707--715.

\bibitem{Suzuki-new-type-of-fp-thm-2009}
  T.~Suzuki,
  \textit{A new type of fixed point theorem in metric spaces},
  Nonlinear Analysis \textbf{71} (2009), 5313--5317.

\end{thebibliography}
\end{document}